\documentclass{amsart}
\usepackage{amssymb}
\usepackage{pb-diagram}

\newtheorem{theorem}{Theorem}[section]

\newtheorem{prop}[theorem]{Proposition}
\newtheorem{lem}[theorem]{Lemma}

\newtheorem{fact}[theorem]{Fact}
\newtheorem{cor}[theorem]{Corollary}

\theoremstyle{definition}
\newtheorem{defn}[theorem]{Definition}

\theoremstyle{remark}

\newcommand{\la}{\langle}
\newcommand{\ra}{\rangle}

\newcommand{\sub}{\subseteq}

\newcommand{\WM}{\widetilde{\cal M}}

\newcommand{\dcl}{\operatorname{dcl}}

\newcommand{\cal}[1]{\ensuremath{\mathcal{#1}}}
\newcommand{\Cal}[1]{\ensuremath{\mathcal{#1}}}

\newcommand{\res}{\ensuremath{\upharpoonright}}

\newcommand{\N}{\mathbb{N}}
\newcommand{\Q}{\mathbb{Q}}
\newcommand{\R}{\mathbb{R}}
\newcommand{\order}{\operatorname{order}}
\newcommand{\graph}{\operatorname{gr}}

\title{The choice property in tame expansions of o-minimal structures}

\begin{document}

\author {Pantelis  E. Eleftheriou}

\address{Department of Mathematics and Statistics, University of Konstanz, Box 216, 78457 Konstanz, Germany}

\email{panteleimon.eleftheriou@uni-konstanz.de}

\thanks{The first author was supported by an Independent Research Grant from the German Research Foundation (DFG) and a Zukunftskolleg Research Fellowship. The second author was partially supported by TUBITAK Career Grant 113F119. The third author was partially supported by NSF grants DMS-1300402 and DMS-1654725.}

\author {Ayhan G\"{u}naydin}

\address{Department of Mathematics, Bo\u{g}azi\c{c}i University, Bebek, Istanbul, Turkey}

\email{ayhan.gunaydin@boun.edu.tr}

\author{Philipp Hieronymi}
\address
{Department of Mathematics\\University of Illinois at Urbana-Champaign\\1409 West Green Street\\Urbana, IL 61801}
\email{phierony@illinois.edu}
\urladdr{http://www.math.uiuc.edu/\textasciitilde phierony}

\begin{abstract}
We establish the choice property, a weak analogue of definable choice, for certain tame expansions of o-minimal structures. Most noteworthily, dense pairs of real closed fields have this property.
\end{abstract}

\date{\today}
 \maketitle

\section{Introduction}



In this paper we establish a weak analogue of definable choice in certain important expansions of o-minimal structures. Throughout  $\Cal M=\la M, <, +, \dots\ra$ is an o-minimal expansion of a densely ordered abelian group whose language is $\Cal L$. Let $\widetilde{\Cal M} = \la\Cal M,P\ra$ be an expansion of $\Cal M$ by a set $P\sub M$ in the language $\Cal L(P)=\Cal L\cup\{P\}$, where we identify the set $P$ with a new unary predicate.
`Definable' means `definable in the language $\Cal L(P)$' and `$\Cal L$-definable' means  `definable in the language $\Cal L$'. When we want to express parameters, we write $A$-definable in the first case and $\Cal L_A$-definable in the second. For a subset $X\subseteq M$, we write $\dcl(X)$ for the definable closure of $X$ in $\Cal M$.
Let $X\sub M^n$ be a definable set. We call $X$ \textbf{small} if there is no $m$ and no $\cal L$-definable function $f:M^{nm}\to M$ such that $f(X^m)$ contains an open interval in $M$.

Pairs $\WM=\la \cal M, P\ra$ with tame geometric behavior on the class of all definable sets have been extensively studied in the literature, and they include dense pairs (\cite{vdd-dense}), expansions of $\cal M$ by a dense independent set (\cite{dms2}), and expansions by a multiplicative group with the Mann property (\cite{dg}). In \cite{egh}, all these examples were put under a common perspective, and a structure theorem was proved for their definable sets, in analogy with the cell decomposition theorem known for o-minimal structures. Namely, after imposing three conditions on the theory of $\WM$ (\cite[Section 2, Assumptions (I)-(III)]{egh}, it was proved  that every definable set is a finite union of \emph{cones}. We need not be  extensive on the results from \cite{egh}, but it is worth pointing out that they imply the failure of definable Skolem functions in that setting.


\begin{fact}\label{noSkolemfunction}(cp. \cite[5.4]{dms1})
Suppose that $\widetilde{\Cal M}=\la\Cal M,P\ra$ satisfies Assumptions (I)--(III) from \cite{egh}.  Let $f: M \to P$ be definable. Then there is a small set $S$ such that $f(M\setminus S)$ is finite. In particular, there is no definable function $h: M \to M$ such that $h(x) \in P\cap (x,\infty)$ for all sufficiently large $x\in M$.
\end{fact}
\begin{proof} Using \cite[Corollary 3.26]{egh} instead \cite[Theorem 3(1)]{vdd-dense}, the same proof as for \cite[5.4]{dms1} works in this case as well.
\end{proof}

 In \cite[Section 5.3]{egh} we introduced the following weak version of definable choice.

\begin{defn}\label{def-cp} Let $h:Z\sub M^{n+k}\to M^l$ be an $\Cal L_A$-definable continuous map and $S\subseteq M^n$ be an $A$-definable small set. We say $\Cal M$ has \textbf{weak definable choice for $(h,S)$} if there are
\begin{itemize}
\item $\Cal L_A$-definable continuous maps $h_1,\dots,h_p$ mapping $M^{m+k}$ into $M^l$,
\item $A$-definable sets $X_1,\dots,X_p\sub M^{m+k}$, and
\item $A$-definable small sets $Y_1,\dots, Y_p\subseteq M^m$,
\end{itemize}
such that for every $a\in \pi(Z)$ and $i=1,\dots,p$,
\begin{enumerate}
\item $X_{ia}\sub Y_i$
\item $h_i(-, a): X_{ia}\to M^l$ is injective, and
\item $h(S \cap Z_a, a)=\bigcup_i h_i(X_{ia}, a)$,
\end{enumerate}
where $\pi(Z)$ denotes the projection of $Z$ onto the last $k$ coordinates.

\noindent
We say that $\widetilde{\Cal M}$ has the \textbf{choice property} if it has weak definable choice for every pair $(h,S)$ as above.
\end{defn}

%

\noindent
Note that in this definition, the sets $Y_i$  could be chosen to be the same small set by taking their union, but we keep it this way as this is how it appears in \cite{egh}.


\medskip
An important consequence of the choice property is that it implies a strong structure theorem (\cite[Theorem 5.12]{egh}), which says that every definable set is a finite \emph{disjoint} union of \emph{strong} cones. This theorem was shown to fail in \cite[Section 5.2]{egh} for general dense pairs, with the counterexample being $\widetilde{\Cal M}=\la \cal M,P\ra$, where $\Cal M=\la\R,<,+,x\mapsto \pi x_{\res [0, 1]}\ra$ and $P=\Q(\pi)$.  In particular, the choice property fails for the above dense pair. 
In \cite[Question 5.13]{egh}  we asked for conditions on $\cal M$ or $\WM$ that guarantee the choice property, and in this paper we establish the following theorem in that regard. 

\begin{theorem}\label{main} Suppose that $\widetilde{\Cal M} = \la\Cal M,P\ra$ satisfies one of the following statements:
\begin{enumerate}
\item $\cal M$ is an ordered $K$-vector space, where $K$ is an ordered field,
\item $P$ is a dense $\dcl$-independent set,
\item $\cal M$ is a real closed field.
\end{enumerate}
Then $\widetilde{\Cal M}$ satisfies the choice property.
\end{theorem}

In Section \ref{sec-easy}, we handle the first two cases.  The bulk of the work is in the last case, which is established in Section \ref{sec-rcf}. In Section \ref{ufss_section}, we study a property equivalent to the choice property and prove some technical results that are used in the last case.

\subsubsection*{Notations and conventions} We will use $i,j,k,l,m,n$ for natural numbers, and $\pi$ always denotes a coordinate projection. Let $X,Y$ be sets. We denote the cardinality of $X$ by $|X|$. If $Z\subseteq X \times Y$ and $x\in X$, then $Z_x$ denotes the set $\{ y\in Y \ : \ (x,y) \in Z\}$.
 For a function $f$, we denote the graph of $f$ by $\graph(f)$. If $f: Z\subseteq X \times Y\to Z'$ and $x\in X$, then $f(x,-)$ denotes the function that maps $y\in Z_x$ to $f(x,y)$. If $a=(a_1,\dots,a_n)$, we sometimes $Xa$ for $X\cup \{a_1,\dots,a_n\}$, and $XY$ for $X\cup Y$.

\section{Vector spaces and dense independent sets}\label{sec-easy}

\subsection{Expansions of vector spaces} Let $K$ be an ordered field and  $\Cal M$ be an ordered $K$-vector space, which is considered as a structure in the language $\Cal L$ of ordered $K$-vector spaces. Recall that the theory of $K$-vector spaces has quantifier elimination in $\Cal L$ (see \cite[Chapter 1]{vdd-book}). It is also well-known that definable functions in $M$ are piecewise affine linear transformations. This is to say that for a definable function $f:X\subseteq M^q \to M$, there is a decomposition of $X$ into semi-linear sets $C_1,\dots, C_t$ such that for each $j=1,\dots, t$ there are $r \in K^q$ and $b \in M$ such that $f(x)=r\cdot x+b$ for all $x\in C_j$, where  $\cdot$ denotes the usual dot-product of tuples with elements in $K$.

\medskip\noindent
Let $P\subseteq M$. We now show that $\la\Cal M,P\ra$ has the choice property. Let $h: Z \subseteq M^{n+k} \to M^l$ be $\Cal L_A$-definable and $S\subseteq M^n$ be an $A$-definable small set. After decomposing $Z$ into finitely many semi-linear sets, we may assume that each $\pi_i\circ h$ is an affine linear function from $M^{n+k}$ to $M$ restricted to the $\Cal L_A$-definable set $Z$. Then there are $r_1,\dots,r_l\in K^n,s_1,\dots,s_l\in K^k$ and $b_1,\dots,b_l\in M$ such that for each $(g,a)\in Z$
\[
h(g,a)=(r_1\cdot g + s_1 \cdot a + b_1,\dots,r_l\cdot g + s_l \cdot a + b_l).
\]

\medskip\noindent
We set
\[
X := \Big\{ (t_1,\dots,t_l,a) \in M^l\times M^k \  : \ \exists g \in S\cap Z_a \bigwedge_{j=1}^l\ r_j \cdot g = t_j\Big\}\text{, and }
\]
\[
Y:=\Big\{(t_1,\dots,t_l)\in M^l:\exists g \in S\bigwedge_{j=1}^l\ r_j \cdot g = t_j\Big\}.
\]
Then $Y$ is small and for each $a\in M^k$ we have $X_a\subseteq Y$.

\medskip\noindent
Let $h_0: M^l \times M^k \to M^l$ map $(t,a)$ to $(t_1 + s_1\cdot a+ b_1,\dots,t_l + s_l\cdot a+ b_l)$. It can be checked easily that $h(S\cap Z_a,a) = h_0(X_a,a)$ and $h_0(-,a)$ is injective for each $a\in \pi(Z)$. This proves the choice property for $\la\Cal M,P\ra$.


\subsection{Expansions by a dense independent set}
Let $\Cal M$ be an o-minimal expansion of an ordered group and let $P$ be a dense $\dcl$-independent subset of $M$. We will show that the pair $\la\Cal M,P\ra$ has the choice property. Before we do so, we recall a bit of notation from \cite{dms2,vdd-book}. We say a set $X \subseteq M^n$ is \textbf{regular} if it is convex in each coordinate, and \textbf{strongly regular} if it is regular and all points in $X$ have pairwise distinct coordinates. A function $f: X \to M$ is called \textbf{regular} if $X$ is regular, $f$ is continuous and in each coordinate, $f$ is either constant or strictly monotone.

\begin{fact}{(\cite[1.5]{dms2} cp. \cite[p.58]{vdd-book})}\label{fact:celldecompo}
Let  $h: Z\subseteq M^m\to M$ be an $\Cal L_A$-definable function. Then there are $\Cal L_A$-definable cells $C_1,\dots,C_r$ such that and
\begin{itemize}
\item[(i)] $Z= \bigcup_{i=1}^r C_i$,
\item[(ii)] if $C_i$ is open, then $C_i$ is strongly regular and the restriction of $h$ to $C_i$ is regular.
\end{itemize}
\end{fact}

\begin{theorem}
$\WM$ has the choice property. Moreover, the sets $Y_i$ from Definition \ref{def-cp} are all equal to $P^m$.
\end{theorem}
\begin{proof}
By \cite[Lemma 3.11]{egh} we may assume that $S\subseteq P^n$. It is also easy to see that it is enough to consider the case when $l=1$. (This fact actually follows from Lemma \ref{lem:l=1} and the proof of Proposition \ref{inj_UFSS}.)

We now prove the Choice Property by induction on $n+k$. When $n+k=0$, the Choice Property holds trivially. So now suppose that $n+k>0$. By Fact \ref{fact:celldecompo} we can assume that $Z$ itself is a  regular cell, $h$ is regular on $Z$, and that if $Z$ is open, then $Z$ is strongly regular. We will first show that we can reduce to the case that $Z$ is open. \\

Suppose that $Z$ is not open. Since $Z$ is a cell, there is a coordinate projection $\sigma : M^{n+k} \to M^{n+k-1}$ that is bijective on $Z$.
Suppose that $\sigma$ misses one of the last $k$ coordinates. By induction the Choice Property holds for $h' : Z' \to M$ and $S$, where $Z' = \sigma(Z)$ and $h':=h\circ \sigma^{-1}$. From this, the weak definable choice for $h$ and $S$ can be deduced easily. Suppose now that $\sigma$ misses one of the first $n$ coordinates. Let $\tau : M^{n} \to M^{n-1}$ be the coordinate projection missing the same coordinate as $\sigma$. Then $\tau(S)\sub M^{n-1}$. By induction the weak definable choice holds for $h' : Z' \to M$ and $\tau(S)$, where $h':=h\circ \sigma^{-1}$ and
\[
Z' := \{ (z,a) \in M^{n-1+k} \ : \ (z,a) \in \sigma(Z), z\in \tau(S)\}.
\]
The weak definable choice for $h$ and $S$ follows easily.\\

We have reduced to the case that $Z$ is open. Thus $Z$ is a strongly regular cell. Using a similar argument as in the case when $Z$ is not open, we can reduce to the case that $h(-,a)$ is strongly regular on $Z_a$ for every $a\in \pi(Z)$. By \cite[1.8]{dms2} and the fact that $S\sub P^n$, we have that for every $a\in \pi(Z)$ and $x\in M$ the set
$\{ y \in S\cap Z_a \ : \ h(y,a) = x\}$ is finite. We define $X\subseteq M^{n+k}$ to be the set of tuples $(g,a) \in Z$ such that $g\in S$ and $g$ is the lexicographic minimum of
\[
\{ y \in S\cap Z_a \ : \ h(y,a) = h(g,a) \}.
\]
The lexicographic minimum always exists, because the set is finite and nonempty. It follows immediately that $X_a \subseteq P^n$, $h(-,a)$ is injective on $X_a$ and $h(S\cap Z_a,a) = h(X_a,a)$.  
\end{proof}

\section{The choice property and uniform families of small sets}\label{ufss_section}
In this section we restate the choice property in terms of definable families of small sets.
The new statement appears to be more natural and it is better suited for the bookkeeping necessary to handle the third case of Theorem \ref{main} in the next section. We also establish several technical facts that will be useful in the sequel.


\begin{defn}  Let $Z\subseteq M^{m+k+l}$ be $\Cal L$-definable, $S \subseteq M^m$  definable and small,  and $X\subseteq M^{m+k}$ definable.  The triple $(Z,S,X)$ is called a \textbf{uniform family of small sets (UFSS)} if for all $a \in M^k$, we have
\begin{enumerate}
\item $X_{a} \subseteq S$, and
\item $Z_{b,a}$ is finite for each $(b,a)\in \pi(Z)$.
\end{enumerate}

\noindent We say that such a family is \textbf{injective} if in addition the following condition holds for each $a\in M^k$:
\begin{itemize}
\item[(3)] $Z_{b,a} \cap Z_{c,a} = \emptyset$ for distinct $b,c \in X_{a}$.
\end{itemize}
For $A\subseteq M$, we say that $(Z,S,X)$ is \textbf{$A$-definable} if $Z$ is $\Cal L_A$-definable, $S$ is $A$-definable and $X$ is $A$-definable.
\end{defn}


\noindent
Observe that the fact $X_a \subseteq S$ guarantees that $\bigcup_{a \in M^k} X_a$ is small. The reason for calling a UFSS as such is that the family $\big(\bigcup_{b \in X_{a}} Z_{b,a}\big)_{a\in M^k}$ becomes $A$-definable and each member is small in a uniform way.

\medskip\noindent
When we say $(Z,S,X)$ is a UFSS and $Z\subseteq M^{m+k+l}$, this will not only mean that $Z\subseteq M^{m+k+l}$, but also that $S\subseteq M^m$ and $X\subseteq M^{m+k}$.

\medskip\noindent
We fix some notation. Let $(Z,S,X)$ be a UFSS with $Z\subseteq M^{m+k+l}$. We say that $(Z,S,X)$ is a \textbf{union of UFSSs} $(Z_1,S_1,X_1),\dots,(Z_p,S_p,X_p)$ if
\[
\bigcup_{b \in X_a} Z_{b,a} = \bigcup_{b \in X_{1,a}} Z_{1,b,a} \cup \dots \cup \bigcup_{b \in X_{p,a}} Z_{p,b,a}
\]
for all $a \in M^k$. Note that the ambient spaces of the sets $Z_i$ might be different than $M^{m+k+l}$, the ambient space of $Z$; likewise for the sets $S_i$ and  $X_i$.\newline

\begin{lem}\label{lem:choiceprop} Let $(Z,S,X)$ be an injective $A$-definable UFSS. Then there is $p \in \N$ and for each $i=1,\dots, p$
there is an $\Cal L_A$-definable continuous map $h_i : Z_i \subseteq M^{m+k}\to M^l$, such that for every $a\in M^k$,
\begin{enumerate}
\item $h_i(-, a): X_{a}\to M^l$ is injective, and
\item $\bigcup_{b \in X_a} Z_{a,b}=\bigcup_{i=1}^{p} h_i(Z_{i,a}\cap X_{a}, a)$,
\end{enumerate}
In particular, $(\graph(h_i),S,X)$ is an injective $A$-definable UFSS for each $i=1,\dots, p$ and $(Z,S,X)$ is a finite union of these UFSSs.
\end{lem}
\begin{proof}
Since $(Z,S,X)$ is a UFSS, $Z_{b,a}$ is finite for each $(b,a) \in \pi(Z)$. Since $Z$ is $\Cal L_A$-definable and $\Cal M$ is o-minimal, there is $q\in \N$ such that $|Z_{b,a}|\leq q$ for all $(b,a) \in \pi(Z)$; given such $(b,a)$ order elements of $Z_{b,a}$ as $y_{b,a,1}\leq y_{b,a,2}\leq \dots\leq y_{b,a,q}$ (if $|Z_{b,a}|=s$, then repeat the smallest element $q-s+1$ times).
Thus there are $\Cal L_A$-definable functions $f_1,f_2,\dots,f_q : \pi(Z) \subseteq M^{n+k} \to M^l$ such that $f_j(b,a)=y_{b,a,j}$ for each $(b,a)\in\pi(Z)$. It is clear from the construction that
\[
\bigcup_{b \in X_a} Z_{b,a}=\bigcup_{j=1}^q f_j(X_{a}, a).
\]
Since $(Z,S,X)$ is injective, it follows immediately that for every $a\in M^k$, each $f_j(-,a)$ is injective. Using cell-decomposition in o-minimal structures, we obtain $p \in \N$ and for each $i=1,\dots,p$  an $\Cal L_A$-definable continuous map $h_i :Z_i \subseteq M^{m+k}\to M^l$ such that $h_i(-,a)$ is injectivefor every $a\in M^k$, and
\[
\bigcup_{i=1}^p h_i(Z_i\cap X_{a}, a)=\bigcup_{i=1}^q f_i(X_{a}, a)=\bigcup_{b \in X_a} Z_{b,a}.
\]
\end{proof}

\medskip\noindent
Next result relates UFSSs with the choice property.

\begin{prop}\label{inj_UFSS} The following are equivalent:
\begin{itemize}
\item[(i)] Every $A$-definable UFSS is a finite union of injective $A$-definable UFSSs,
\item[(ii)] $\widetilde{\Cal M}$ has the choice property.
\end{itemize}
\end{prop}

\begin{proof}
(i)$\Rightarrow$(ii): Let $h:Z\sub M^{n+k}\to M^l$ be an $\Cal L_A$-definable continuous map and $S\subseteq M^n$ be an $A$-definable small set. Set
\[
W := \{ (y,x,h(y,x)) \ : \ (y,x) \in Z\} ,  \quad X := Z\cap(S\times M^k).
\]
Then it is immediate to check that $(W,S,X)$ is an $A$-definable UFSS. By our assumption, there are injective $A$-definable UFSSs  $(Z_1,S_1,X_1),\dots,(Z_p,S_p,X_p)$ such that $(W,S,X)$ is union of these UFSSs. Thus
\[
h(S\cap X_a,a) = \bigcup_{b \in X_a} W_{b,a} = \bigcup_{b \in X_{1,a}} Z_{1,b,a} \cup \dots \cup \bigcup_{b \in X_{p,a}} Z_{p,b,a}
\]
We can easily modify $S_1,\dots,S_p$ such that there is $m\in \N$ with $S_i\subseteq M^m$ for all $i=1,\dots,p$. By Lemma \ref{lem:choiceprop} each of the $\bigcup_{b \in X_{i,a}} Z_{i,b,a}$ is of the desired form.\newline
(ii)$\Rightarrow$(i): Let $(Z,S,X)$ be an $A$-definable UFSS with $Z\subseteq M^{m+k+l}$. By cell decomposition and since $Z_{b,a}$ is finite for every $(b,a)\in M^{m+k}$, we may assume that there is an $\Cal L_A$-definable continuous function $h : \pi(Z) \to M^l$ such that $Z=\graph(h)$. By the choice property we get an $\Cal L_A$-definable continuous maps $h_1,\dots,h_p$ mapping $M^{m+k}$ into $M^l$, $A$-definable sets $X_1,\dots,X_p\sub M^{m+k}$, and $A$-definable small sets $Y_1,\dots, Y_p\subseteq M^m$, such that for every $a\in \pi(Z)$ and $i=1,\dots,p$,
\begin{enumerate}
\item $X_{ia}\sub Y_i$
\item $h_i(-, a): X_{ia}\to M^l$ is injective, and
\item $h(S \cap Z_a, a)=\bigcup_i h_i(X_{ia}, a)$,
\end{enumerate}
Now for each $i=1,\dots,p$ define
\[
X'_i := \{ (x,a) \in X_i \ : \ \exists y\in M^m \ (y,a) \in X \wedge h(y,a)=h_i(x,a)\}.
\]
It is straightforward to see that by (1) and (2) the triples $(\graph(h_1),Y_1,X_1'),\dots,$ $(\graph(h_p),Y_p,X_p')$ are injective $A$-definable UFSSs. By (3) and the definition of $X_i'$, we have that $(\graph(h),S,X)$ is a union of these UFSSs.
\end{proof}

\noindent We now collect a few easy lemmas about UFSSs that are helpful showing that in a given structure every UFSS is a finite union of injective UFSSs.

\begin{lem}\label{lem:sub} Let $S \subseteq M^n$ be small and let $(Z_1,S,X), (Z_2,S,X)$ be $A$-definable UFSSs, where $Z_2 \subseteq M^{n+k+l}$. If
\begin{itemize}
\item $Z_{1,a} \subseteq Z_{2,a}$ for all $a\in M^k$,
\item $(Z_2,S,X)$ is a finite union of injective $A$-definable UFSSs,
\end{itemize}
then $(Z_1,S,X)$ is a finite union of injective $A$-definable UFSSs.
\end{lem}
\begin{proof}
Suppose there are injective $A$-definable UFSSs $(W_1,S_1,X_1),\dots, (W_p,S_p,X_p)$ such that $(Z_2,S,X)$ is a union of these UFSSs.
By Lemma \ref{lem:choiceprop} we may assume that $|W_{i,b,a}|=1$ for $i=1,\dots, p$ and $(b,a)\in M^{m_i+k}$ where $W_i\subseteq M^{m_i+k+l}$. Now define
\[
Y_i := \{ (b,a) \in X_i \ : \ \exists c \in X_a \ W_{i,b,a} \subseteq Z_{1,c,a}\}.
\]
It is easy to check that each $(W_i,S_i,Y_i)$ is an injective  UFSS, because $(W_i,S_i,X_i)$ is. From our definition of $Y_i$ it follows easily that $(Z_1,S,X)$ is a union of the injective UFSSs $(W_1,S_1,Y_1),\dots, (W_p,S_p,Y_p)$.
\end{proof}

\begin{lem}\label{lem:l=1} If every UFSS, $(W,S,Y)$ with $W\subseteq M^{m+k+1}$ is a finite union of injective $A$-definable UFSSs, then $\widetilde{\Cal M}$ has the choice property.
\end{lem}

\begin{proof}
By Proposition \ref{inj_UFSS}, it suffices to show that every UFSS is a finite union of injective $A$-definable UFSSs.
So let $(Z,S,X)$ be a UFSS with $Z\subseteq M^{m+k+l}$ where $l>1$.  For $i=1,\dots,l$, let $Z_i'=\pi_i (Z)\subseteq M^{m+k+1}$ where $\pi_i$ is the projection onto the first $m+k$ coordinates and the $m+k+i$-th coordinate. It is clear that each $(Z_i',S,X)$ is an $A$-definable UFSS; hence by assumption, it is a finite union of injective $A$-definable UFSSs.

\medskip\noindent
We define
\[
W:= \{ (b,a,x) \in M^{m+k+l} \ : \ \pi_i(b,a,x) \in Z'_{i} \hbox{ for } i=1,\dots, l\}.
\]
We observe that for each $(b,a)\in \pi(Z)$, we have that $Z_{b,a}\subseteq W_{b,a}$ and $W_{b,a}$ is finite, since $Z_{b,a}$ is finite. Therefore $(W,S,X)$ is a UFSS. By Lemma \ref{lem:sub} it is left to show that $(W,S,X)$ is a finite union of injective $A$-definable UFSSs.\newline

\noindent Now for each $i=1,\dots,l$, let $(Z_{i,1},S_{i,1},X_{i,1}),\dots,(Z_{i,p},S_{i,p},X_{i,p})$ be injective $A$-definable UFSSs such that $(Z_i',S,X)$ is a union of these UFSSs. Without loss of generality we can assume that the same $p$ works for all $i$. For $\sigma : \{1,\dots,l\} \to \{1,\dots,p\}$ we define
\begin{align*}
Z_{\sigma} &:= \{ (b_1,\dots,b_l,a,z_1,\dots,z_l) \ : \ (b_i,a,z_i) \in Z_{i,\sigma(i)} \hbox{ for } i=1,\dots,l\}\\
S_{\sigma} &:= S_{1,\sigma(1)} \times \dots \times S_{l,\sigma(l)}\\
X_{\sigma} &:= \{ (b_1,\dots,b_l,a) \ : \ (b_i,a) \in X_{i,\sigma(i)} \hbox{ for } i=1,\dots,l\}.
\end{align*}
It is easy to check that each $(Z_{\sigma},S_{\sigma},X_{\sigma})$ is an injective $A$-definable UFSS and that for each $a \in M^k$
\[
\bigcup_{b\in X_a} W_{b,a} = \bigcup_{\sigma : \{1,\dots,l\} \to \{1,\dots,p\}} \bigcup_{c \in X_{\sigma,a}} Z_{\sigma,c,a}.
\]

%
%




\end{proof}

\begin{lem} Let $S\subseteq M^n$ be small and $A$-definable, and $Z\subseteq M^{n+1}$ be an $\Cal L_A$-definable cell such that $\dim Z_x=0$ for each $x\in \pi(Z)\subseteq M^n$. Then there is an $A$-definable small set $S'$ such that
\begin{itemize}
\item[(1)] $\bigcup_{g\in S} Z_g = \bigcup_{h\in S'} Z_h$.
\item[(2)] $Z_{h_1}\cap Z_{h_2} = \emptyset$ for $h_1,h_2\in S'$ with $h_1\neq h_2$.
\end{itemize}
\end{lem}
\begin{proof}
Since $Z$ is a cell and $\dim Z_x =0$ for each $x\in \pi(Z)$, we have that $|Z_x|=1$ for each $x \in \pi(Z)$. By definable choice in o-minimal structures, there is an $\Cal L_A$-definable function $f: M^n \to M^n$ such that for each $x,y\in M^n$
\begin{itemize}
\item $Z_{f(x)} = Z_x$, and
\item $Z_{f(x)} \cap Z_{f(y)}=\emptyset$ whenever $f(x)\neq f(y)$.
\end{itemize}
Note that $f(S)$ is small and $A$-definable. Therefore the conclusion holds with $S':=f(S)$.
\end{proof}

\begin{cor}\label{cor:cpcase0} Let $(Z,S,X)$ be an $A$-definable UFSS such that $Z\subseteq M^{m+k+l}$. If $k=0$, then $(Z,S,X)$ is a finite union of injective $A$-definable UFSSs.
\end{cor}

\noindent We collect two more lemmas whose very easy, but technical proofs we leave for the reader.

\begin{lem}\label{lem:indstep1} Let $(Z,S,X)$ be an $A$-definable UFSS such that $Z\subseteq M^{m+k+l}$, and let  $f : M^{m} \to M^{n}$ be $\Cal L_A$-definable. If $(Z,S,X)$ is a finite union of $A$-definable injective UFSSs, then so is $(Z',S',X')$ where
\begin{align*}
Z' &:= \{ (b,f(b),a,c) \in M^{m+n+k+l} : \ (b,a,c) \in Z\},\\
S' &:= \{(b,f(b)) : b\in S \},\\
X' &:= \{ (b,f(b),a) \in M^{m+n+k} \ : \ (b,a) \in X\}.
\end{align*}
\end{lem}

\begin{lem}\label{lem:indstep2} Let $(Z,S,X)$ be an $A$-definable UFSS such that $Z\subseteq M^{m+k+l}$, and let $f : M^{m+k} \to M^{n}$ be $\Cal L_A$-definable. If $(Z,S,X)$ is a finite union of injective $A$-definable UFSSs, then so is $(Z',S,X')$ where
\begin{align*}
Z' := \{ (b,a,f(b,a),c) \in M^{m+k+n+l} : \ (b,a,c) \in Z\}\\
X' := \{ (b,a,d) \in M^{m+k+n} \ : \ (b,a) \in X,\, d=f(a,b)\}.
\end{align*}
\end{lem}

\noindent
Note that in this lemma, $m$ and $l$ are preserved and $k$ is replaced with $k+n$.

\section{Expansions of real closed fields}\label{sec-rcf}

Let $\Cal M$ be a real closed field. Let $P$ a subset of $M$. In this section, we will show that $\widetilde{\Cal M}=\la\Cal M,P\ra$ has the choice property. We start by fixing some notation we will use in the proof.\newline

\noindent Define the following  order on $\N^k \times \N$: $(i_1,\dots,i_k,r) \prec (j_1,\dots,j_k,s)$ if and only if one of the following two conditions holds:
\begin{itemize}
\item $i_1+\dots + i_k + r < j_1+\dots+j_k+s$,
\item $i_1+\dots + i_k + r = j_1+\dots+j_k+s$ and $(i_1,\dots,i_k,r) <_{lex} (j_1,\dots,j_k,s)$.
\end{itemize}
Observe that $(\N^k\times \N,\prec)$ has order type $\omega$. We denote the order isomorphism between $\N$ and $\N^k \times \N$ by $\sigma$.\newline

\noindent Let $K$ be a field and consider the polynomial ring $K[X_1,\dots,X_k,Y]$ in $n+1$ variables. For $p \in K[X_1,\dots,X_k,Y]$, the $\order(p)$ is the $\prec$-maximal element $(i_1,\dots,i_k,r)$ in $\N^k \times \N$ such that the monomial
\[
X_1^{i_1}\cdots X_k^{i_n} Y^s
\]
appears with a non-zero coefficient in the polynomial $p$.

\medskip\noindent
Now we are ready to prove the first part of Theorem \ref{main}.

\begin{theorem}
$\WM$ has the choice property.
\end{theorem}

\begin{proof} Let $(Z,S,X)$ be an $A$-definable UFSS where $Z\subseteq M^{n+k+l}$. By Proposition \ref{inj_UFSS}, it suffices to show that $(Z,S,X)$ is finite union of injective $A$-definable UFSSs. We proceed by induction on $k$.
The case $k=0$ is just Corollary \ref{cor:cpcase0}.
\newline

So now suppose that $k>0$ and the statement holds $k'<k$. By quantifier elimination for real closed fields, we can assume that $Z$ is a finite union of sets of the form
\begin{equation}\label{eq:proof}
\tag{$\ast$}\big\{ (b,a,c) \in M^{n+k+l} \ : \ p(b,a,c) = 0, q_1(b,a,c)>0,\dots, q_s(b,a,c)>0, (a,b) \in U\big\}.
\end{equation}
where $p,q_1,\dots,q_s$ are polynomials in $\Q(A)(x_1,\dots,x_n)[y,z]$ and $U$ is some $\Cal L_A$-definable set. We can directly reduce to the case that $Z$ is of the form \eqref{eq:proof}.
By Lemma \ref{lem:sub}, we can reduce to the case that $q_1=\dots=q_s=1$. By Lemma \ref{lem:l=1}, we can assume that $l=1$.\newline

\noindent We now show the following statement:

\noindent \begin{itemize}
\item Let $\alpha\in \N^k \times \N$. If $(Z,S,X)$ is an $A$-definable UFSS such that there is $p \in \Q(A)(x)[y,z]$, and $\Cal L_A$-definable set $U$ such that
\begin{enumerate}
\item $Z =\{ (b,a,c) \in M^{n+k+1} \ : \ p(b,a,c) = 0, (b,a) \in U\}$
\item $\order(p(b,-,-))\preceq \alpha$ for every $b\in M^k$,
\end{enumerate}
then $(Z,S,X)$ is a finite union of injective $A$-definable UFSSs.
\end{itemize}

\noindent We prove this statement by induction on $\alpha$ with respect to the well-order $\prec$. So let $\alpha \in \N^k\times \N$, and let $(Z,S,X)$ be an $A$-definable UFSS, $p(x,y,z)\in \Q(A)(x)[y,z]$ and $U\subseteq M^{n+k}$ $\Cal L_A$-definable such that (1) and (2) hold.  \newline

\noindent Let $f_{i,j} : M^n \to M$ be rational functions over the field $\Q(A)$ such that
\[
p(x,y,z) = \sum_{(i,j) \in \N^{k+1}} f_{i,j}(x) y^i z^j.
\]
Let $I$ be the finite set of all $(i,j)\in \N^{k+1}$ such that $f_{i,j} \neq 0$.   For each $(i,j) \in I$ define
\[
W_{i,j} := \{ b \in M^n \ : \ \order\big(p(b,-,-)\big) = (i,j)\}
\]
Observe that $W_{i,j}$ is $\Cal L_A$-definable. We can directly reduce to the case that $Z \subseteq W_{v,w}\times M^{k+1}$ for some $(v,w)\in I$. By replacing some of the $f_{i,j}$'s by $0$, we can further assume that $(v,w)$ is the $\prec$-maximum of $I$. By dividing $p$ by $f_{v,w}(x)$, we can assume that $f_{v,w}(x)=1$ for every $x \in W_{v,w}$.

\medskip\noindent
Let $n_1,\dots,n_{|I|}\in \N$ such that $\sigma^{-1}(I)=\{n_1,\dots,n_{|I|}\}$\footnote{Recall that $\sigma$ is the order isomorphism between $(\N,<)$ and $(\N^k \times \N,\prec)$.}. Define $h : M^n \to M^{|I|}$ to be the function given by
\[
x \mapsto (f_{\sigma(n_1)}(x),\dots,f_{\sigma(n_{|I|})}(x)).
\]
For $d=(d_{n_1},\dots,d_{n_{|I|}})\in M^{|I|}$, let $q_d$ denote the polynomial
\[
q_d(y,z) := \sum_{(i,j)\in I} d_{\sigma^{-1}((i,j))} y^i z^j.
\]
Set $S_0 := h(S)$. Observe that $S_0$ is small, since $S$ is. Since $f_{\sigma(n_{|I|})}(b) = 1$ for all $b \in S$, we get $\order(q_{h(b)}) = \order(p(b,-,-))\preceq \alpha$.
Define
\begin{align*}
Z_0 &:= \{ (d,a,c) \in M^{|I|+k+1} \ : \ q_d(a,c) = 0, \ q_d(a,-)\neq 0\}\\
X_0 &:= \{ (h(b),a) \ : \ b \in X_a, a\in M^k \}.
\end{align*}
Observe that for each $a \in M^k$ we have
\[
\bigcup_{d \in X_{0,a}} Z_{0,a,d} \supseteq \bigcup_{b \in X_a} Z_{a,b}.
\]
By Lemma \ref{lem:sub} it is enough to check that $(Z_0,S_0,X_0)$ is a finite union of injective $A$-definable UFSSs.\newline


\noindent



\noindent
Let $W$ be the $\Cal L_A$-definable set
\[
\{ (d_1,d_2,a,c) \in M^{2|I|+k+1} \ : \ d_1 \neq d_2 \wedge (d_1,a,c)\in Z_0 \wedge (d_2,a,c) \in Z_0\}.
\]
Define $X_1 \subseteq M^{|I|+k+1}$ by
\[
\{ (d,a)\in S_0 \times M^k \ : \ \forall d' \in S_0 \ (\exists c \in M \ (d,a,c)\in Z_0 \wedge (d',a,c) \in Z_0) \rightarrow (d=d')\}.
\]
Observe that for all $a\in M^k$
\[
\bigcup_{d \in X_{0,a}} Z_{0,d,a} = \bigcup_{d\in X_{1,a}} Z_{0,d,a} \cup \bigcup_{(d_1,d_2) \in X_{0,a}^2} W_{d_1,d_2,a,c}.
\]
Therefore it is enough to show that both $(Z_0,S_0,X_1)$ and $(W,S_0^2,X_0^2)$ are finite unions of injective $A$-definable UFSSs. It follows directly from the definition of $X_1$ that $(Z_0,S_0,X_1)$ is an injective UFSS. It is only left to consider $(W,S_0^2,X_0^2)$.\newline

\noindent We now show that $(W,S_0^2,X_0^2)$ is a finite union of injective $A$-definable UFSSs. For $(d_1,d_2) \in M^{2|I|}$, let $r_{d_1,d_2}(y,z)$ be the polynomial $q_{d_1}(y,z)-q_{d_2}(y,z)$.
Because $d_{1,|I|}=d_{2,|I|}=1$ for every $d_1,d_2 \in S_0$, we have that $\order(r_{d_1,d_2})\prec \order(q_{d_1})\preceq \alpha$ for every $d_1,d_2 \in S_0$.\newline

\noindent We split up $W$ into
\begin{align*}
V_1 &:= \{ (d_1,d_2,a,c) \in W \ : \ r_{d_1,d_2}(a,-) \neq 0\}\\
V_2 &:= W \setminus V_1.
\end{align*}
It is left to show that both $(V_1,S_0^2,X_0^2)$ and $(V_2,S_0^2,X_0^2)$ are finite unions of injective $A$-definable UFSSs.\newline

\noindent We first prove this for $(V_1,S_0^2,X_0^2)$. For the following let $\pi : M^{2|I|+k+1} \to M^{2|I|+k}$ be the coordinate projection onto the first $2|I|+k$ coordinates. Let
\[
U_1 :=\{ (d_1,d_2,a,c) \in M^{2|I|+k+1} \ : \ r_{d_1,d_2}(a,c)=0, (d_1,d_2,a) \in \pi(V_1) \}.
\]
Observe that $U_1 \supseteq V_1$. By Lemma \ref{lem:sub} it is enough to show that $(U_1,S_0^2,X_0^2)$ is a finite union of injective UFSSs. Since $r_{d_1,d_2}(a,-)\neq 0$, we have that $U_{1,d_1,d_2,a}$ is finite for each $(d_1,d_2,a) \in \pi(U)$. Since $\order(r_{d_1,d_2})\prec \alpha$, $(U_1,S_0^2,X_0^2)$ is a finite union of injective $A$-definable UFSSs by the induction hypothesis.\newline

\noindent We now consider $(V_2,S_0^2,X_0^2)$. Let
\[
U_2 : =\{ (d_1,d_2,a,c) \in M^{2|I|+k+1} \ : \ r_{d_1,d_2}(a,-)=0\}.
\]
Observe that $r_{d_1,d_2}(-,-)\neq 0$ whenever $d_1\neq d_2$. Therefore $\dim \pi(U_2)<2|I|+k$. It follows easily from Lemma \ref{lem:indstep1} and Lemma \ref{lem:indstep2} and our induction hypothesis on $k$ that $(V_2,S_0^2,X_0^2)$ is a finite union of injective UFSSs.
\end{proof}

\end{document}